\documentclass[10pt]{article}
\font\smallit=cmti10
\font\smalltt=cmtt10

\usepackage{amssymb,latexsym,amsmath,epsfig,amsthm} 

\makeatletter

\renewcommand\section{\@startsection {section}{1}{\z@}
{-30pt \@plus -1ex \@minus -.2ex}
{2.3ex \@plus.2ex}
{\normalfont\normalsize\bfseries}}

\renewcommand\subsection{\@startsection{subsection}{2}{\z@}
{-3.25ex\@plus -1ex \@minus -.2ex}
{1.5ex \@plus .2ex}
{\normalfont\normalsize\bfseries}}

\renewcommand{\@seccntformat}[1]{\csname the#1\endcsname. }

\makeatother

\newtheorem{theorem}{Theorem}
\newtheorem{lemma}{Lemma}

\newtheorem{corollary}{Corollary}

\newcommand{\beq}{\begin{equation}}
\newcommand{\eeq}{\end{equation}}

\def\({\left(}
\def\){\right)}

\begin{document}

\begin{center}
\uppercase{\bf On the relationship between the number of solutions of congruence systems and the resultant of
two polynomials}
\vskip 20pt
{\bf Dmitry I. Khomovsky}\\
{\smallit Lomonosov Moscow State University, Moscow, RF}\\
{\tt khomovskij@physics.msu.ru}
\end{center}
\vskip 30pt

\centerline{\smallit Received: , Revised: , Accepted: , Published: } 
\vskip 30pt

\centerline{\bf Abstract}

\noindent
Let $q$ be an odd prime and $f(x)$, $g(x)$ be polynomials with integer coefficients. If the system of congruences $f(x) \equiv g(x) \equiv 0 \pmod{q}$ has $\ell$ solutions, then $R\(f(x),g(x)\)\equiv 0 \pmod{q^\ell}$, where $R\(f(x),g(x)\)$ is the resultant of the polynomials. Using this result we give new proofs of some known congruences involving the Lucas sequences.

\pagestyle{myheadings}
\markright{\smalltt INTEGERS: 15 (2015)\hfill}
\thispagestyle{empty}
\baselineskip=12.875pt
\vskip 30pt

\section{Introduction} The resultant $R\(f,g\)$  of two polynomials $f(x)=a_n x^n+ \cdots +a_0$ and $g(x)=b_m x^m+ \cdots
+b_0$ of degrees $n$ and $m$, respectively, with coefficients in a field $F$ is defined by the determinant of  the
$(m+n)\times (m+n)$ Sylvester matrix \cite{4}  \beq \label{1}R\(f,g\)=\begin{vmatrix}
 a_{n} & a_{n-1} & \cdots & \cdots & \cdots& a_0 &  &   &   &   \\
   & a_{n} & a_{n-1}& \cdots & \cdots & \cdots & a_0 &  &   &   \\
  &  & \cdots &  &  &    &    &   &  & \\
  &  &   & a_{n}&a_{n-1} & \cdots  & \cdots  & \cdots &a_0\\
b_{m} & b_{m-1}& \cdots & \cdots & b_0    &  &   &  &  \\
 & b_{m} & b_{m-1}& \cdots & \cdots  & b_0 &  &  &  \\
  &  & \cdots &  &  &    &    &   &  & \\
  &  &   & \cdots&  &    &    &   &  &\\
 & &   & &b_{m}&b_{m-1} & \cdots  & \cdots  &b_0
\end{vmatrix}. \eeq Let $f, g, h$ and $v$ be the polynomials below. Some important properties of the resultant are:

$(i)$ If $f(x)=a_n \prod_{i=1}^{n}(x-\alpha_{i})$ and $g(x)=b_m \prod_{j=1}^{m}(x-\beta_{j})$, then \beq \label{2}
R\(f,g\)=a_{n}^m \prod_{i=1}^{n}g(\alpha_i)=(-1)^{mn}b_{m}^n \prod_{i=1}^{m}f(\beta_i)=a_{n}^m b_{m}^n
\prod_{i=1}^{n}\prod_{j=1}^{m}(\alpha_{i}-\beta_{j}),\nonumber\eeq where $\alpha_{i}$ and $\beta_{j}$ are the roots of $f(x)$ and $g(x)$, respectively, in some extension of $F$, each repeated according to its multiplicity. This property is often taken as the definition of the resultant.\\

$(ii)$ $f$ and $g$ have a common root in some extension of $F$ if and only if $R\(f,g\)=0$.\\

$(iii)$ $R\(f ,g\)=(-1)^{n m}R\(g,f\)$.\\

$(iv)$ $R\(f h,g\)=R\(f,g\)R\(h,g\)$ and $R\(f,g h\)=R\(f,g\)R\(f,h\)$.\\

$(v)$ If $g=v f+h$ and $\deg(h)=d$, then $R\(f,g\)=a_{n}^{m-d}R\(f,h\)$.\\

$(vi)$ If $p$ is a positive integer, then $R\(f(x^p),g(x^p)\)=R\(f(x),g(x)\)^p$.\\

\noindent All these properties are well-known \cite{5,2}. More details concerning the resultant can be found in \cite{3,1}. Another important classical result is (see \cite{1}):

\begin{lemma} Let $f=\sum_{i=0}^{n} a_i x^i$ and $g=\sum_{j=0}^{m} b_j x^j$ be two polynomials of degrees $n$ and $m$,
respectively. Let, for $k\geq0$, $r_{k}(x)=r_{k,n-1}x^{n-1}+\cdots+r_{k,0}$ be the remainder of $x^k g(x)$ modulo $f(x)$,
i.e., $x^k g(x)=v_{k}(x)f(x)+r_{k}(x)$, where $v_{k}$ is some polynomial and $\deg(r_k)\leq n-1$. Then \beq \label{3}
R\(f,g\)=a_{n}^m \begin{vmatrix}
 r_{n-1,n-1} & r_{n-1,n-2} & \cdots & r_{n-1,0}   \\
 r_{n-2,n-1} & r_{n-2,n-2} & \cdots & r_{n-2,0}   \\
      \vdots &             &        & \vdots   \\
 r_{0,n-1}   & r_{0,n-2}   & \cdots & r_{0,0}
\end{vmatrix}. \eeq
\end{lemma}
In the next section we prove a theorem on the relationship
between the number of solutions of the congruence system $f(x) \equiv g(x) \equiv 0 \pmod{q}$  and  the resultant of two
polynomials $R\(f(x),g(x)\)$. Then using this result we give new proofs of some congruences involving the Lucas sequences.

\section{Properties of the resultant}

A polynomial $f(x)$ with integer coefficients is called {\it not identically zero in $\mathbb{Z}_q$}  if at least one of its coefficients is not divisible by $q$. Let $A=\(a_{i,j}\)$ be an arbitrary matrix. Then by $A^{<q>}$ we will denote the matrix $\(a_{i,j}'\)$ over $\mathbb{Z}_q$ of the same type such that $a_{i,j}'$ is the residue of $a_{i,j}$ modulo $q$.

\label{a}\begin{theorem} Let $q$ be a prime and $f(x)$, $g(x)$ be polynomials with integer coefficients that are not identically zero in $\mathbb{Z}_q$. If the system of congruences $f(x) \equiv 0 \pmod{q}$ and $g(x) \equiv 0 \pmod{q}$ has $\ell$ solutions, then $R\(f(x),g(x)\)\equiv 0 \pmod{q^\ell}$.\end{theorem}

\begin{proof} Let $\deg{f}=n$ and $\deg{g}=m$. Then we have that the system $f(x) \equiv g(x) \equiv 0 \pmod{q}$ has
$\ell$ solutions by the theorem conditions and $\ell\leq\min[n,m]$
as the polynomials are not identically zero in $\mathbb{Z}_q$.
Let $r_{k}(x)=r_{k,n-1}x^{n-1}+\cdots+r_{k,0}$ be the remainder of $x^k g(x)$ modulo $f(x)$, i.e., $x^k
g(x)=v_{k}(x)f(x)+r_{k}(x)$, where $v_{k}(x)$ is some polynomial and $\deg(r_k)\leq n-1$. Then we get the system of
congruences \beq \label{4} \begin{pmatrix}
 r_{n-1,n-1} & r_{n-1,n-2} & \cdots & r_{n-1,0}   \\
 r_{n-2,n-1} & r_{n-2,n-2} & \cdots & r_{n-2,0}   \\
      \vdots &             &        & \vdots   \\
 r_{0,n-1}   & r_{0,n-2}   & \cdots & r_{0,0}
\end{pmatrix} \begin{pmatrix} x^{n-1}\\ x^{n-2} \\
 \vdots \\
1 \end{pmatrix} \equiv \begin{pmatrix} 0 \\ 0 \\
 \vdots \\
0 \end{pmatrix}\pmod{q}. \eeq This system has at least $\ell$ solutions, since each congruence of
$(\ref{4})$ is derived from $f(x) \equiv 0 \pmod{q}$ and $g(x) \equiv 0 \pmod{q}$. Let $A=\(a_{i,j}\)$ be a matrix of the
system $(\ref{4})$.  With the help of the procedure analogous  to row reduction using operations of swapping the rows and adding a multiple of one row to another row, we can reduce $A$ to a matrix $A_1$ with integer coefficients such that
$\det{\(A\)}=\pm\det{\(A_1\)}$ and $A_1^{<q>}$ is an upper triangular matrix. We can see that each solution of the system
$(\ref{4})$ is also a solution of the following system over $\mathbb{Z}_q$: \beq \label{5} \(A_1^{<q>}\) \begin{pmatrix}
x^{n-1}\\ x^{n-2} \\
 \vdots \\
1 \end{pmatrix} \equiv \begin{pmatrix}
 \vdots \\
0 \end{pmatrix}\pmod{q}, \eeq so $(\ref{5})$ has at least $\ell$ solutions. Note that the last $\ell$ congruences of $(\ref{5})$ have degrees less than $\ell$. On the other hand, these congruences have at least $\ell$ solutions. Hence all these congruences must be congruences with zero coefficients, i.e., the last $\ell$ rows of $A_1^{<q>}$ are zero rows. Therefore, all elements of the last $\ell$ rows of $A_1$ are divisible by $q$, so $\det{\(A\)}=\pm\det{\(A_1\)}$ is divisible by $q^{\ell}$. Thus, by Lemma $1$ we have $R\(f,g\)\equiv 0 \pmod{q^\ell}$.
\end{proof}

\noindent {\bf Remark.} If one or both polynomials equal zero in $\mathbb{Z}_q$, then by property $(i)$ we obtain that either $R(f,g)\equiv 0\pmod{q^n}$ or $R(f,g)\equiv 0\pmod{q^m}$.  We do not consider this trivial case in Theorem 1.

\noindent{\bf Example.} Let $f(x)=x^{6}+1,\, g(x)=(x+1)^{6}+1$. The system of congruences $x^{6}+1\equiv 0 \pmod{13}$ and $(x+1)^{6}+1\equiv
0 \pmod{13}$ has three solution in $\mathbb{Z}_{13}$: $x=5,6,7$. The matrix of the system $(\ref{4})$ for these polynomials is:
\beq \label{6} A=\left( \begin{array}{cccccc}
 1 & -6 & -15 & -20 & -15 & -6 \\
 6 & 1 & -6 & -15 & -20 & -15 \\
 15 & 6 & 1 & -6 & -15 & -20 \\
 20 & 15 & 6 & 1 & -6 & -15 \\
 15 & 20 & 15 & 6 & 1 & -6 \\
 6 & 15 & 20 & 15 & 6 & 1
\end{array} \right). \eeq Since the resulting echelon form of matrices after row reduction  is not unique, we obtain the reduced row echelon form of the matrix $A$, which is unique: \beq \label{6.1} A_1^{<13>}=\left( \begin{array}{cccccc}
 1 & 7 & 11 & 6 & 11 & 7 \\
 0 & 1 & 9 & 11 & 1 & 11 \\
 0 & 0 & 1 & 8 & 3 & 11 \\
 0 & 0 & 0 & 0 & 0 & 0 \\
 0 & 0 & 0 & 0 & 0 & 0 \\
 0 & 0 & 0 & 0 & 0 & 0
\end{array} \right) \longrightarrow
\left( \begin{array}{cccccc}
 1 & 0 & 0 & 7 & 4 & 8 \\
 0 & 1 & 0 & 4 & 0 & 3 \\
 0 & 0 & 1 & 8 & 3 & 11 \\
 0 & 0 & 0 & 0 & 0 & 0 \\
 0 & 0 & 0 & 0 & 0 & 0 \\
 0 & 0 & 0 & 0 & 0 & 0
\end{array} \right).\eeq So we get $\det{A}\equiv 0 \pmod{13^3}$ and $R\(x^{6}+1,(x+1)^{6}+1\)\equiv 0 \pmod{13^3}$. This
resultant is actually equal to $2^4\times5\times13^3$.

\begin{corollary} Let $q$ be a prime and $f(x)$, $g(x)$ be polynomials of degrees $n$ and $m$, respectively, with integer coefficients that are not identically zero in $\mathbb{Z}_q$. Let $A$ be a matrix  of the system $(\ref{4})$ for $f(x), g(x)$. If $\operatorname{Rank} {A}=p$ in $\mathbb{Z}_q$, then $R\(f,g\)\equiv 0\pmod{ q^{n-p}}$. If the system $f(x) \equiv g(x) \equiv 0 \pmod{q}$ has $\ell$ solutions, then $n-p\geq \ell$. Moreover, if  $M$ is any $k\times k$ minor of the matrix $A$ and $k> p$, then $M\equiv 0\pmod{ q^{k-p}}$.
\end{corollary}
\begin{proof} This follows from Theorem 1.\end{proof}

The question about the relation of the multiplicity of $q$ as a factor of $R(f,g)$ and the degree of common
factor of the polynomials $f$ and $g$ modulo $q$  was studied in \cite{10}. This question is closely related to Theorem 1 and first appeared in \cite{11}.
\section{The congruences involving the terms of the Lucas sequences}

\begin{theorem} Let $f(x)=a_n x^n+ \cdots +a_0$ be a polynomial of degree $n$ with  integer coefficients and $q$ be an odd prime. Let $a_0\not\equiv 0\pmod{q}$ and let the congruence $f(x)\equiv 0 \pmod{q}$ have $\ell$ solutions. Then \beq \label{7} R(f(x),x^{q-1}-1)\equiv a_n^{q-1}\prod_{i=1}^{n}(\alpha_i^{q-1}-1)\equiv 0 \pmod{q^\ell}, \eeq where $\alpha_{i}$ are the roots of $f(x)$, each repeated according to its multiplicity. \end{theorem}

\begin{proof} Consider $R(f(x),x^{q-1}-1)$. Since $f(x)=a_n \prod_{i=1}^{n} (x-\alpha_{i})$, then \beq \label{8}
R\(f(x),x^{q-1}-1\)=a_n^{q-1} \prod_{i=1}^{n}(\alpha_i^{q-1}-1).\eeq We know that $q$ is an odd prime, so the congruence
$x^{q-1}-1\equiv 0 \pmod{q}$ has $q-1$  solutions (zero is not one of them). On the other hand, the congruence $f(x)\equiv 0 \pmod{q}$ has
$\ell$ nonzero solutions, as $a_0\not\equiv 0 \pmod{q}$. Hence the system of congruences $f(x)\equiv x^{q-1}-1\equiv
0 \pmod{q}$ also has  $\ell$ solutions. Then by Theorem 1  we have $R(f(x),x^{q-1}-1)\equiv 0 \pmod{q^\ell}$.\end{proof}

\begin{theorem} Let $f(x)=a_n x^n+ \cdots +a_0$ be a polynomial of degree $n$ with  integer coefficients and $q$ be an odd prime. Let $a_0\not\equiv 0\pmod{q}$ and let the congruence $f(x)\equiv 0 \pmod{q}$ have $\ell$ solutions. If $b$ solutions of them are quadratic residues modulo $q$, then
\beq \label{9} R(f(x),x^{\frac{q-1}{2}}-1)\equiv a_n^{\frac{q-1}{2}}\prod_{i=1}^{n}(\alpha_i^{\frac{q-1}{2}}-1)\equiv 0\pmod{q^b} \eeq and \beq \label{10} R(f(x),x^{\frac{q-1}{2}}+1)\equiv
a_n^{\frac{q-1}{2}}\prod_{i=1}^{n}(\alpha_i^{\frac{q-1}{2}}+1)\equiv 0 \pmod{q^{\ell-b}}, \eeq where $\alpha_{i}$ are the roots of $f(x)$, each repeated according to its multiplicity.\end{theorem}
\begin{proof} Consider $R(f(x),x^{\frac{q-1}{2}}-1)$. Since $f(x)= a_n \prod_{i=1}^{n}(x-\alpha_{i})$, then \beq \label{11}
R\(f(x),x^{\frac{q-1}{2}}-1\)=a_n^{\frac{q-1}{2}}\prod_{i=1}^{n}(\alpha_i^{\frac{q-1}{2}}-1).\eeq We know that $f(x)\equiv 0 \pmod{q}$ has $b$ nonzero solutions which are quadratic residues modulo $q$. Hence the  system of congruences $f(x)\equiv x^{\frac{q-1}{2}}-1\equiv 0 \pmod{q}$ has $b$ solutions. Then by Theorem 1  we have $R(f(x),x^{q-1}-1)\equiv 0 \pmod{q^b}$.
Since $\ell-b$ solutions of $f(x)\equiv 0 \pmod{q}$ are quadratic nonresidues modulo $q$, then by analogy we prove that $R(f(x),x^{\frac{q-1}{2}}+1)\equiv 0\pmod{q^{\ell-b}}$.
\end{proof}

As an illustration of applications of Theorem $1$ we consider the following  theorem.

\begin{theorem} Let $q$ be an odd prime and  $P$, $Q$ be any integers such that $Q\not\equiv 0\pmod{q}$. If the  Legendre symbol
$\left(\frac{P^2-4Q}{q}\right)$ is equal to $1$, then \beq \label{12} V_{q-1}(P,Q)\equiv Q^{q-1}+1 \pmod{q^2},\eeq \beq
\label{13}V_{\frac{q-1}{2}}^2(P,Q)\equiv \(Q^\frac{q-1}{2}+1\)^2 \pmod{q^2},\eeq where $V_{n}(P,Q)$ is the $n$-th  term of the Lucas sequence defined by the recurrence relation
\beq
\label{13.1}V_0=2,\,\,\, V_1=P,\,\,\, V_{i}=P V_{i-1}-Q V_{i-2},\,\,\, i\geq2.
\eeq
\end{theorem}
\begin{proof} The roots of $x^2-P x+Q$ are $\alpha_{1}=\frac{P-\sqrt{P^2-4Q}}{2}$,
$\alpha_{2}=\frac{P+\sqrt{P^2-4Q}}{2}$. Hence
$R(x^2-Px+Q,x^{q-1}-1)=(\alpha_{1}\alpha_{2})^{q-1}-(\alpha_{1}^{q-1}+\alpha_{2}^{q-1})+1=1+Q^{q-1}-V_{q-1}(P,Q)$. Since we know $\left(\frac{P^2-4Q}{q}\right)=1$ and $Q\not\equiv 0\pmod{q}$, then the system of congruences
$x^2-P x+Q\equiv x^{q-1}-1\equiv 0 \pmod{q}$ has two solutions. So by Theorem 1 we have $1+Q^{q-1}-V_{q-1}(P,Q)\equiv 0 \pmod{q^2}$, thus we get $(\ref{12})$.
Now using the identity $V_{2n}(P,Q)= V_{n}^2(P,Q)-2Q^n$, we obtain $(\ref{13})$.\end{proof}
Note that the congruences $(\ref{12})$ and $(\ref{13})$ are well-known \cite{7,8,9}, but here we give an alternative completely independent proof of these results.

\begin{corollary} Let $q$ be an odd prime and $k$, $P$, $Q$ be any integers such that $k^2+P k+Q\not\equiv 0\pmod{q}$. If  $\left(\frac{P^2-4Q}{q}\right)=1$, then \beq \label{14} V_{q-1}(P+2k,k^2+P k+Q)\equiv (k^2+Pk+Q)^{q-1}+1 \pmod{q^2},\eeq
\beq \label{15}V_{\frac{q-1}{2}}^2(P+2k,k^2+P k+Q)\equiv \((k^2+P k+Q)^\frac{q-1}{2}+1\)^2 \pmod{q^2}.\eeq
\end{corollary}
\begin{proof} Since $(P+2k)^2-4(k^2+P k+Q)=P^2-4Q$, this corollary follows from Theorem 4.\end{proof}


\subsection{The congruences involving the Lucas numbers}

Let $P=1$, $Q=-1$ and $\left(\frac{5}{q}\right)=1$, i.e., by  the Quadratic Reciprocity Law $q\equiv \pm 1\pmod{5}$. Let an integer $k$ satisfy $k^2+ k-1\not\equiv 0\pmod{q}$, then by Corollary $2$ \beq \label{L1} V_{q-1}(1+2k,k^2+k-1)\equiv
(k^2+k-1)^{q-1}+1 \pmod{q^2},\eeq
\beq \label{L2}V_{\frac{q-1}{2}}^2(1+2k,k^2+k-1)\equiv
(k^2+k-1)^{q-1}+2(k^2+k-1)^{\frac{q-1}{2}}+1 \pmod{q^2}.\eeq
If $k=0$, then \beq \label{L3} L_{q-1}\equiv 2 \pmod{q^2},\eeq
\beq
\label{L4}L_{\frac{q-1}{2}}^2\equiv 2+2(-1)^{\frac{q-1}{2}} \pmod{q^2}, \eeq
where $L_{n}$ is the $n$-th  Lucas number.

\subsection{The congruences involving the Pell-Lucas numbers.}

Let $P=2$, $Q=-1$ and $\left(\frac{8}{q}\right)=1$, i.e., by the Quadratic Reciprocity Law $q\equiv \pm 1\pmod{8}$. Let an integer $k$ satisfy $k^2+2k-1\not\equiv 0\pmod{q}$, then by Corollary $2$
\beq \label{P1} V_{q-1}(2+2k,k^2+ 2k-1)\equiv
(k^2+2k-1)^{q-1}+1 \pmod{q^2},\eeq
\beq \label{P2}V_{\frac{q-1}{2}}^2(2+2k,k^2+2k-1)\equiv
(k^2+2k-1)^{q-1}+2(k^2+2k-1)^{\frac{q-1}{2}}+1 \pmod{q^2}.\eeq
If $k=0$, then
\beq \label{P3} \widetilde{P}_{q-1}\equiv 2 \pmod{q^2},\eeq
\beq \label{P4} \widetilde{P}_{\frac{q-1}{2}}^2\equiv 2+2(-1)^{\frac{q-1}{2}} \pmod{q^2}, \eeq
where $\widetilde{P}_{n}$ is the $n$-th  Pell-Lucas number defined by:
\beq \label{P5} \widetilde{P}_0=2,\,\,\, \widetilde{P}_1=2,\,\,\, \widetilde{P}_i=2\widetilde{P}_{i-1}+\widetilde{P}_{i-2},\,\,\, i\geq2. \eeq

\noindent{\bf Acknowledgments.} The author would like to thank Prof. K.A. Sveshnikov and Prof. R.M.  Kolpakov  for valuable suggestions. Also, the author is indebted to the referee for many useful comments.

\medskip



\begin{thebibliography}{99}
\bibitem{5} P.M. Cohn, {\it Algebra, Vol. 1}, Wiley, New York, 1980.
\bibitem{10} D. G\'{o}mez-P\'{e}rez, J. Gutierrez, A. Ibeas and D. Sevilla. {\it Common factors of resultants modulo
$p$}, Bull. Aust. Math. Soc. {\bf 79} (2009), 299-302.
\bibitem{3} C. Helou, G. Terjanian, {\it Arithmetical properties of wendt's determinant}, Journal of Number Theory
\textbf{115} (2005), 45--57.
\bibitem{1} S. Janson, {\it Resultant and discriminant of polynomials}, Notes, September, \textbf{22} (2007).
\bibitem{11} S.  V.  Konyagin  and  I.  Shparlinski. {\it Character  Sums  with  Exponential  Functions  and  their
Applications}, Cambridge University Press (1999).
\bibitem{7} R. J. McIntosh, E. L. Roettger, {\it A search for Fibonacci-Wieferich and Wolstenholme primes}, Math. Comp. \textbf{76} (2007), 2087–-2094.
\bibitem{2} P. Ribenboim, {\it Fermat's Last Theorem for Amateurs}, Springer, New York, 1999.
\bibitem{8} Z.-H. Sun, Z.-W. Sun, {\it Fibonacci numbers and Fermat's last theorem}, Acta Arith \textbf{60} (1992),
371-–388.
\bibitem{9} Z.W. Sun, R. Tauraso, {\it New congruences for central binomial coefficients}, Adv. in Appl. Math.
\bibitem{4} B.L. van der Waerden, {\it Algebra, vol. 1}, F. Ungar Pub. Co., New York, 1977.




\end{thebibliography}
\end{document}